\documentclass[11pt]{amsart}
\usepackage{comment}
\newtheorem{theorem}{Theorem}[section]
\newtheorem{corollary}{Corollary}[section]
\newtheorem{lemma}{Lemma}[section]

\newtheorem{remark}{Remark}[section]

\newtheorem{definition}{Definition}[section]
\newtheorem{proposition}{Proposition}[section]
\usepackage{color}
\usepackage{soul}

\def \a{\alpha }

\def \l{\lambda }

\begin{document}

\newcommand\nat{\mathbb N}
\newcommand\ganz{\mathbb Z}

\newcommand{\wta}{{\rm {wt} }  a }
\newcommand{\R}{\Bbb  R}

\newcommand{\wtb}{{\rm {wt} }  b }
\newcommand{\bea}{\begin{eqnarray}}
\newcommand{\eea}{\end{eqnarray}}
\newcommand{\be}{\begin {equation}}
\newcommand{\ee}{\end{equation}}
\newcommand{\g}{\frak g}
\newcommand{\hg}{\hat {\frak g} }
\newcommand{\hn}{\hat {\frak n} }
\newcommand{\h}{\frak h}
\newcommand{\V}{\Cal V}
\newcommand{\hh}{\hat {\frak h} }
\newcommand{\n}{\frak n}
\newcommand{\Z}{\Bbb Z}
\newcommand{\N}{{\Bbb Z} _{> 0} }
\newcommand{\Zp} {\Z _ {\ge 0} }
\newcommand{\Hp}{\bar H}
\newcommand{\C}{\Bbb C}
\newcommand{\Q}{\Bbb Q}
\newcommand{\1}{\bf 1}
\newcommand{\la}{\langle}
\newcommand{\ra}{\rangle}
\newcommand{\NS}{\bf{ns} }

\newcommand{\wt}{{\rm {wt} }   }

\newcommand{\E}{\mathcal E}
\newcommand{\F}{\mathcal F}
\newcommand{\X}{\bar X}
\newcommand{\Y}{\bar Y}

\newcommand{\hf}{\mbox{$\tfrac{1}{2}$}}
\newcommand{\thf}{\mbox{$\tfrac{3}{2}$}}

\newcommand{\W}{\mathcal{W}}
\newcommand{\non}{\nonumber}
\def \l {\lambda}
\baselineskip=14pt
\newenvironment{demo}[1]%
{\vskip-\lastskip\medskip
  \noindent
  {\em #1.}\enspace
  }%
{\qed\par\medskip
  }

\def \l {\lambda}
\def \a {\alpha}

\title[On parafermion vertex  algebras ]{On parafermion vertex  algebras of $\frak{sl}(2)_{-3/2}$ and $\frak{sl}(3)_{-3/2}$}

\author[]{Dra\v zen  Adamovi\' c}
\author[]{Antun Milas}
\author[]{Qing Wang}

\date{\today}
\begin{abstract}
We study parafermion  vertex algebras $N_{-3/2}(\frak{sl}(2))$ and $N_{-3/2}(\frak{sl}(3))$. Using  the isomorphism between
$N_{-3/2}(\frak{sl}(3))$ and the logarithmic vertex algebra  $\mathcal{W}^{0} (2)_{A_2} $  from \cite{A-TG},
we show that these parafermion vertex algebras are infinite direct sums of irreducible modules for the Zamolodchikov algebra $\mathcal{W}(2,3)$   of central charge $c=-10$, and that $N_{-3/2}(\frak{sl}(3))$ is a direct sum of irreducible $N_{-3/2}(\frak{sl}(2))$--modules.   
 As a byproduct, we prove  certain conjectures about  the vertex algebra $\mathcal{W}^0(p)_{A_2}$. We also obtain a vertex-algebraic proof of the irreducibility of a family of $\mathcal W(2,3)_{c}$ modules at $c=-10$.
\end{abstract}
\maketitle

\section{Introduction}

Parafermion vertex algebras, or parafermionic cosets, are closely connected to the so-called $Z$-algebras \cite{LP,LW1,LW2}, which played an important role in the development of the parafermion conformal field theory \cite{G,ZF}  and in the vertex-operator-theoretic interpretation of Rogers-Ramanujan type partition
identities \cite{LW1,LW2}. They are first defined
in \cite{DL} as the subalgebras of the generalized vertex algebras generated by $Z$-operators.

In the past decades, the parafermion vertex operator algebras associated to rational affine vertex operator algebras at the positive integer level were thoroughly studied and their structure is well-understood(see \cite{ALY1,ALY2,DLY,DLWY,DR,DW1,DW2,DW3,KP,JW,Wang} etc.). At other levels, including
generic levels, their structure is largely unknown.  Recently, these algebras have appeared in \cite{CFL, L-17} as  quotients of certain universal vertex algebras constructed from non-linear conformal algebras.

Among them,
a very interesting problem is to determine fusion rules for parafermion vertex algebras at rational level (cf.  \cite{ACR}).

Our general goal is to study parafermion vertex algebra beyond rational case. It is natural to start with certain examples when a parafermion vertex algebra  belongs to a certain class of $\mathcal W$-algebras. Let us mention that in the case of $\frak{sl}(2)$, parafermion
algebra $N_k(\frak{sl}(2))$ coincides with singlet vertex algebra for $k=-\frac12, -\frac43 $(cf. \cite{Wa},  \cite{A-JPAA}),  with super-singlet (cf. \cite{A-TG}, \cite{ACR}, \cite{AMP} ) for $k=-\frac23$. In this paper, we want to explore the parafermion vertex algebra at a certain non-integral admissible level, which belongs to the class of logarithmic vertex algebras. More specifically, we study the following vertex algebras:

 \begin{itemize}
 \item The parafermion vertex algebras  $N_k (\frak{sl}(2))$  and $N_k (\frak{sl}(3))$ at level $k=-\frac32$;
 \item Higher rank logarithmic vertex algebras $ \mathcal{W}(p)_{A_2}$ and $\mathcal{W}^0(p)_{A_2}$ for $p=2$;
 \item
   The  (universal) principal $W$--algebra $\mathcal W(2,3)_c = \mathcal  W^k(\frak{sl}(3), f_{pr})$, also known as
 Zamolodchikov's algebra, at central charge $c=-10$.
  \end{itemize}
 In \cite{A-TG}, it was proved that $N_{-3/2} (\frak{sl}(3)) \cong \mathcal{W}^0(2)_{A_2}$, where $\mathcal{W}(2)_{A_2}$ is a logarithmic vertex algebra (the so-called "octuplet" algebra) constructed from the lattice vertex algebra $V_{\sqrt{2}A_2}$, and where $\mathcal{W}^0(2)_{A_2}$ is the zero charge subalgebra of $\mathcal{W}(2)_{A_2}$ \cite{FT,S1}.

 One of our main results is the following theorem.
 \begin{theorem}  \leavevmode
\makeatletter
\@nobreaktrue
\makeatother
 \begin{itemize}
 \item[(1)]  $\mathcal{W}^0(2)_{A_2}$ and $\mathcal{W}(2)_{A_2} $ are completely reducible  $N_{-3/2} (\frak{sl}(2))$--modules.
  \item[(2)]  $\mathcal{W}^0(2)_{A_2} $ and $\mathcal{W}(2)_{A_2} $ are completely reducible  $\mathcal{W}(2,3)_{c=-10}$--modules.
  \item[(3)] $N_{-3/2} (\frak{sl}(2))$ is a completely reducible $\mathcal W(2,3)_{c=-10}$--module.
  \item[(4)] $N_{-3/2} (\frak{sl}(3))$ is generated by primary vectors of weights $3,4,4$.
 \end{itemize}
 \end{theorem}

Note that $N_{-3/2} (\frak{sl}(2))$ is conformally embedded into    $\mathcal W  ^0(2) _{A_2}$,
which  is a parafermionic analog of the conformal embedding $\frak{gl}(2) \hookrightarrow \frak{sl}(3) $ at $k=-\frac32$ investigated in \cite{AKMPP-16}. We prove in Proposition \ref{invarijant-1} the result:

\begin{itemize}
\item $\mathcal W   ^{0}  (2)_{A_2} = \bigoplus_{s \in {\Z}} N_{-3/2} (2 j \omega_1).$
\end{itemize}

 Our next goal is to determine the decomposition of parafermion algebras as $\mathcal W(2,3)$--modules. Fortunately, the characters of $\mathcal{W}^0(2)_{A_2}$ is known, and from its expression one can conjecture the decomposition of $\mathcal{W}^0(2)_{A_2}$ as $\mathcal W(2,3)$--modules.  But there are two problems:
 \begin{itemize}
 \item It seems that  the character formula for  irreducible $\mathcal W(2,3)$--modules has not  been rigorously proved (cf.   Remark \ref{koos}).

 \item One needs to identify singular  vectors for  $\mathcal{W}^0(2)_{A_2}$ and its subalgebras.

 \end{itemize}
Although the primary goal of the paper is     not the study of the  algebra $\mathcal W (2)_{A_2} $, we need to use some elements of its representation theory to solve the above mentioned problems.
 It turns out that the most efficient tool is to use $\frak{sl}(3)$--action on $\mathcal W (2)_{A_2} $.

 A general Lie algebra action on $\mathcal W(p)_Q$ was conjectured by  Feigin-Tipunin \cite{FT}, and recently was proved by S. Sugimoto \cite{Su}. In the case  of the triplet vertex algebra, the $\frak{sl}(2)$--action was obtained in   \cite{AdM-triplet},  \cite{ALM}.
So we need this action for $\frak{sl}(3)$ and $p=2$. Next problem is that in general,  it is still not proved that  $\mathcal W(p)_Q$ is a simple vertex algebra, so one can not prove a semi-simplicity result by applying quantum Galois theory. But in our case we can prove that
$\mathcal W(2)_{A_2}$ is simple by applying   the explicit realization of $L_{-3/2}(\frak{sl}(3))$ from \cite{A-TG} and using identification of the subalgebra  $\mathcal W ^0 (2)_{A_2}$ as a parafermion subalgebra $N_{-3/2}(\frak{sl}(3))$  for which we know that it  is simple.
We get:
  \begin{theorem}

  \item[(1)] The vertex algebra $\mathcal W(2)_{A_2}$ is simple.

 \item[(2)] $\mathcal W(2)_{A_2} $ is a completely reducible $\frak{sl}(3) \times  \mathcal W(2,3)_{c=-10}$--module and the following decomposition holds
 $$ \mathcal W(2)_{A_2} = \bigoplus _{\lambda \in P_+ \cap Q} V_{A_2} (\lambda) \otimes  T^{\kappa =1/2} _{\lambda, 0}.$$

 \item[(3)]  The $\mathcal W(2,3)_{c=-10}$--module $T^{\kappa =1/2} _{\lambda, 0}$ is irreducible.
  \end{theorem}

We apply this result on the decomposition of  $N_{-3/2} (\frak{sl}(2))$ as $\mathcal{W}(2,3)_{c=-10}$--modules. First we show in Proposition \ref{invarijant-1}:

\begin{itemize}
\item $N_{-3/2} (\frak{sl}(2)) =\left( \mathcal W  (2) _{A_2}  \right)^{\frak{gl}(2)}$.
\end{itemize}

As a consequence we get:

\begin{theorem}
$N_{-3/2} (\frak{sl}(2))  = \bigoplus _{\lambda \in P_+ \cap Q} T^{\kappa =1/2} _{\lambda, 0}.$
\end{theorem}

\vskip 5mm

{\bf Acknowledgments:}
This work was partially done during the visit of D.A. to Xiamen in January 2019, and during  the conference Representation Theory XVI in Dubrovnik in June 2019.

 D.A. is  partially supported  by the
QuantiXLie Centre of Excellence, a project cofinanced
by the Croatian Government and European Union
through the European Regional Development Fund - the
Competitiveness and Cohesion Operational Programme
(KK.01.1.1.01.0004). A.M. was partially supported by the NSF Grant DMS-1601070. Q.W.  is partially supported by China NSF grants (Nos.11531004, 11622107).

\section{Preliminaries}

\subsection{Settings and known facts}
In this part we setup some notation and summarize facts we need later.

\begin{itemize}
\item Let $\g $ be the simple Lie algebra with Cartan subalgebra $\h$ and triangular decomposition $\g = \frak n_{-} + \frak h + \frak n_{+}$.
\item Let $\hg$ be the associated affine Lie algebra, and $\hh$ be the associated Heisenberg subalgebra.
\item Let $V^k(\g)$ be the universal affine vertex operator algebra of level $k$ associated to the simple Lie algebra $\g$.
\item Let $L_k(\g)$ be the simple quotient of  $V^k(\g)$.
\item Let $N^k(\g)= \{ v \in V^k(\g) \ \vert \ h(n) v =0 \ \ h \in {\h}, n \in {\Z}_{\ge 0} \}$ be the parafermion subalgebra of $V^k(\g)$.
\item  Let $N_k(\g)= \{ v \in L_k (\g) \ \vert \ h(n) v =0 \ \ h \in {\h}, n \in {\Z}_{\ge 0} \}$ be the parafermion subalgebra of $L_k(\g)$.
\item For a $\lambda \in P_+$, let $V_{\g}(\mu)$ be the irreducible finite-dimensional $\g$--module with the highest weight $\lambda$, where $P_{+}$ denotes the set of dominant integral weights for $\g$.
\item Let $V^k(\lambda)$ be the generalized Verma module for $\hg$--module induced from $\g$--module $V_{\g}(\lambda)$. Let $L_k(\lambda)$ be its simple quotient.
\item   For $\lambda, \mu \in  P_+$, let $T ^{k+3} _{\lambda,\mu}$ denotes the $\mathcal W^k (\g, f_{pr})$--module obtained as
$ H_{DS} (  V^{k} (\lambda - (k+3) \mu) $ (cf. \cite{AF}).
\item   For $k= -3 + \frac{1}{p}$ and $\g=\frak{sl}(3)$,   the universal affine  vertex algebra $V^k(\frak{sl}(3))$ is simple (cf. \cite{GK}), and therefore  by \cite{Arakawa}  $   H_{DS} (  V^{k} (\frak{sl}(3)) = \mathcal  W^k(\frak{sl}(3), f_{pr})$ is simple. In particular, $\mathcal W(2,3)_{c=-10}$ is a simple vertex algebra.
\end{itemize}
We shall need the following facts which are well-known. Let $\g =\frak{sl}(2)$  with a Chevalley basis $\{e,f,h\}$ and let $k=-\frac32$. Then we have:
\begin{itemize}
\item $V^k(\g) = L_k(\g)$.
\item $V^k(j \omega_1) = L_k(j \omega_1)$, $j \in {\Z}_{\ge 0}$,  where $\omega_1$ is the fundamental dominant weight for $\frak{sl}(2)$.

\item $N^k(j):=N^k(j \omega_1) = N_k(j \omega_1)$, $j \in {\Z}_{\ge 0}$, where $N^k(j) =\{ v \in   V^k (j \omega_1) \ \vert \ h(n) v = 0, \ \forall \ n \in {\Z}_{\ge 0} \}$ and $N_k(j) =\{ v \in   L_k(j \omega_1) \ \vert \ h(n) v = 0, \ \forall \ n \in {\Z}_{\ge 0} \}$.
\end{itemize}
We denote by ${\rm ch}[M](q):={\rm tr}_M q^{L(0)}$  the character of a $V$-module $M$; from the context it should be clear what the vertex algebra is. Also, for simplicity we suppressed the conformal anomaly $-\frac{c}{24}$.

\subsection{ The vertex algebra $\mathcal W(2,3)_{c}$ }

\label{prel-1}
Let $\mathcal W(2,3)_{c}$ denotes the principal affine $W$--algebra  $W^k(\frak{sl}(3), f_{pr})$ of central charge $c=c_k =2 - 24\frac{(k+2)^2}{k+3}$ \cite{Arakawa}. It is generated by the Virasoro field $L(z) = \sum_{m\in {\Z} }  L(m) z^{-m-2}$ and another field of conformal
weight $3$:
$$W(z) = \sum _{m \in {\Z}} W(m) z^{-m-3} $$
satisfying bracket relations
\begin{align*}
[L(m),W(n)]&=(2m-n)W(m+n) \\
[W(m),W(n)]&=\frac{(22+5c)c}{48 \cdot 3 \cdot 5!} (m^2-4)(m^2-1) m \delta_{m+n,0} \\
&+\frac13(m-n)\Lambda_{m+n}+\frac{(22+5c)(m-n)}{48 \cdot 30}(2m^2-mn+2n^2-8)L(m+n),
\end{align*}
where $\Lambda=L(-2)^2{\bf 1}-\frac35 L(-4){\bf 1}$.
 Let
$L^{W}(c, h, h_W)$ denotes the irreducible highest weight $\mathcal W(2,3)_{c}$--module of the highest weight $(h,h_W)$ with respect to $(L(0), W(0))$.

Next we discuss characters of modules for the $W(2,3)$ vertex algebra at  $c=-10$.
Out strategy is to first give an upper bound for graded dimensions of a family of irreducible  $W(2,3)$-modules parametrized with dominant integral weights.

We closely follow Arakawa-Frenkel's paper here \cite{AF}. We let $\kappa=3+k$, where $k$ is the level.
Also, for $\lambda \in P_+$, weight lattice of $\frak{sl}(3)$, we denote by
 ${{V}}^\kappa({\lambda})$ the Weyl module with the top degree isomorphic to $V_{\frak{sl}(3)}(\lambda)$. Then in \cite{AF}, a family of modules for the universal $W(2,3)$-algebra are defined $T_{\lambda,\mu}^\kappa=H_{DS}^0({V}^\kappa({\lambda-\kappa \mu}))$.
We only consider modules with $\mu=0$. For given central character determined by $\lambda$, we denote by $\beta_{\lambda}$ the $W(0)$ eigenvalue on the highest weight vector.

For two $q$-series $f$ and $h$, we write $f(q) \leq h(q)$ if  $${\rm Coeff}_{q^n} f(q) \leq {\rm Coeff}_{q^n} h(q),$$
for every $n$. Then results from \cite{AF} give
\begin{lemma} \label{irr-upper}
For every $m,n \in \mathbb{Z}_{\geq 0}$, we have:
$$ {\rm ch}[L^{W}(-10, \frac23m^2+\frac23n^2+\frac23 mn+m+n, \beta_{m,n})](q)
$$
$$\leq \frac{q^{\frac23m^2+\frac23n^2+\frac23 mn+m+n}(1-q^{m+1})(1-q^{n+1})(1-q^{m+n+2})}{(q;q)_\infty^2}.$$
\end{lemma}
\begin{proof}
We apply formula (5.7)   from \cite{AF} for the character formula of $T_{\lambda,0}^\kappa$, with $\kappa=\frac{1}{2}$
and $\frak{g}=\frak{sl}(3)$. For $\lambda=m \omega_1 + n \omega_2$, $m,n \geq 0$, their formula is
$${\rm ch}[T_{\lambda,0}^{\kappa}](q)= \frac{ q^{\tilde{\Delta}_{\lambda,0}^{\kappa}} \sum_{w \in W} (-1)^{\ell(w)} q^{-\langle w(\lambda+\rho),\rho\rangle}}{(q;q)_\infty^2}, $$
where
$$\tilde{\Delta}_{\lambda,0}^{\kappa}=\frac{1}{2 \kappa} (\lambda, \lambda +2 \rho)+(\rho,\rho).$$
Plugging in $\lambda=m \omega_1 + n \omega_2$, and summing over the Weyl group $W$, gives
$${\rm ch}[T_{\lambda,0}^{\kappa}](q)=\frac{q^{\frac23m^2+\frac23n^2+\frac23 mn+m+n}(1-q^{m+1})(1-q^{n+1})(1-q^{m+n+2})}{(q;q)_\infty^2}.$$
%This statement was already mentioned in \cite{FT}, as a part of more general result (see also \cite{KV} for discussion in %the case of $\mathcal W(2,3)$).  Here we only provide some details.
Since $T_{\lambda,0}^{\kappa}$ is not necessarily irreducible for $\kappa$ rational (it is always irreducible for
$\kappa \notin \mathbb{Q}$ \cite{AF}) there might be non-trivial maximal submodule in $T_{\lambda,0}^{\kappa}$ so we  conclude
$$\mbox{ch}[L^{W}(-10, \frac23m^2+\frac23n^2+\frac23 mn +m+n, \beta_{m,n})](q) \leq  {\rm ch}[T_{\lambda,0}^{\kappa}](q)$$
as claimed.

\end{proof}

\section{The Vertex algebra $\mathcal{W} (p)_{A_2}$ and its companions }

  In this section we recall the definition of vertex algebras $ \mathcal{W}(p)_{A_2}$ and  $ \mathcal{W}^0(p)_{A_2}$  (cf. \cite{AdM-peking}, \cite{S1} ) which   are   a higher analog of the triplet vertex algebra and singlet vertex algebra  (cf. \cite{A-singlet}, \cite{AdM-singlet} \cite{AdM-triplet}). We shall also recall the result of \cite{A-TG} which identifies $\mathcal{W}^0 (p)_{A_2}$ for $p=2$, as a parafermionic vertex algebra $N_{-3/2}(\frak{sl}(3))$.

In this part we closely follow \cite{A-TG}. We consider the integral lattice
$$ \sqrt{p} A_2
 = {\Z} \gamma_1 + {\Z} \gamma_2, \quad \langle \gamma_1, \gamma_1\rangle = \langle \gamma_2, \gamma_2 \rangle = 2 p,
 \ \langle \gamma_1, \gamma_2 \rangle = - p, $$
 and the  associated lattice vertex algebra $V_{\sqrt{p} A_2}$.
 Let $M_{\gamma_1, \gamma_2} (1)$ be the  Heisenberg vertex subalgebra of
 $V_{ \sqrt{p} A_2}$ generated by the Heisenberg fields $\gamma_1 (z)$ and $\gamma_2 (z)$.
 Let $$\omega_{st} = \frac{1}{3p} ( \gamma_1 (-1) ^2    +  \gamma_1 (-1)  \gamma_2(-1)    +   \gamma_2 (-1) ^ 2 )$$ be the standard Virasoro vector in the lattice vertex algebra $V_{\sqrt{p}A_2}$ of central charge $2$. Define a new conformal vector
 $$\omega = \omega_{st} + \frac{p-1}{p} (  \gamma_1 (-2)   +   \gamma_2  (-2) ).$$
We equip $V_{\sqrt{p}A_2}$ with the conformal structure coming from $\omega$, which has central charge $c_p=2-24 \frac{(p-1) ^{2}}{p}$, e.g. $c_{2}=-10$ for $p=2$.

 The vertex algebra $\mathcal{W}(p)_{A_2}$ is defined (cf. \cite{AdM-peking}, \cite{S1} ) as a subalgebra of the lattice vertex algebra $V_{ \sqrt{p} A_2}$ realized as
$$ \mathcal{W}(p)_{A_2} = \mbox{Ker} _{ V_{ \sqrt{p} A_2} } e^{-\gamma_1 /p } _0 \bigcap
 \mbox{Ker} _{ V_{ \sqrt{p} A_2} } e^{-\gamma_2 /p } _0 . $$
We also have its subalgebra:
$$ \mathcal{W} ^ 0 (p) _{A_2}  = \mbox{Ker} _{ M_{\gamma_1, \gamma_2} (1) } e^{-\gamma_1 /p } _0 \bigcap
 \mbox{Ker} _{ M_{\gamma_1, \gamma_2} (1) } e^{-\gamma_2 /p } _0    $$

   One can also construct the following extension of $\mathcal{W}^0(p)_{A_2}$ which is a higher rank analog of the doublet vertex algebra $\mathcal A(p)$ from \cite{AdM-doublet}.

  Define the lattice $\Gamma = {\Z}\delta_1 + {\Z} \delta_2 \supset \sqrt{p} A_2$ such that
  $$\delta_1 = \frac{1}{3} ( 2\gamma_1 + \gamma_2), \delta_2 = \frac{1}{3} ( \gamma_1 +  2 \gamma_2). $$
  Clearly, $\Gamma = \sqrt{p} P$, where $P$ is a weight lattice of $A_2$.  $V_{\Gamma}$  has the structure of a generalized vertex algebra (cf. \cite{DL})  which contains  the lattice vertex algebra $V_{\sqrt{2} A_2}$.
  Note that $\gamma_1, \gamma_2 $ belongs to the dual lattice of $\Gamma$, so screening operators $ e^{-\gamma_i /p } _0$ are well defined on $V_{\Gamma}$.
  \begin{definition} We define:
  $$ \Omega  (p) _{A_2} = \mbox{Ker} _{ V_{ \Gamma } } e^{-\gamma_1 /p } _0 \bigcap
 \mbox{Ker} _{ V_{\Gamma}  } e^{-\gamma_2 /p } _0 . $$
 \end{definition}
Then  $ \Omega  (p) _{A_2}$ is a  generalized vertex algebra. It is a vertex algebra for $p \equiv 0 \mbox{mod} (3)$. We have the following inclusions:
$$ \mathcal W(2,3)_{c_{p}} \subset \mathcal W^0  (p) _{A_2}   \subset \mathcal W  (p) _{A_2} \subset  \Omega  (p) _{A_2}. $$

 $\mathcal{W}  (p)_{A_2}$ and  $\mathcal W  ^ 0 (p) _{A_2}  $ have  vertex subalgebras isomorphic to the simple
$\mathcal{W}(2,3)$--algebra  with central charge $c_p$   which is  generated by $\omega$ and
\bea  w_{3}=&&  \frac{1}{p ^3} ( \gamma_2(-1) ^ 3 +  \frac{3}{2} \gamma_1( -1) \gamma_2 (-1) ^2  - \frac{3}{2} \gamma_1( -1)^2  \gamma_2 (-1)  - \gamma_1(-1) ^ 3   ) \nonumber  \\
&-&   \frac{9 (p-1) }{4 p ^3} (2 \gamma_1 (-1) \gamma_1 (-2) + \gamma_1 (-2) \gamma_2 (-1) -   \gamma_2 (-2) \gamma_1 (-1) -2 \gamma_2 (-1) \gamma_2 (-2)  )  \nonumber  \\
&+&    \frac{9 (p-1) ^2 }{2 p ^3} (\gamma_2 (-3) - \gamma_1(-3)) \nonumber.
 \eea
The overall normalization of $w_{3}$ is not important. For example, in order to get bracket relations as in  Section  \ref{prel-1}, for $c=-10$, we would have to consider $\frac{4\sqrt{2}}{27}w_{3}$.

  By direct calculation we get:
 \begin{proposition} \label{T} Assume that $\lambda = m\omega_1 + n \omega_2$.  Let $\kappa = \frac{1}{p}$. Let
% $T^{\kappa}_{\lambda, 0}$ is isomorphic to the $W(2,3)$--module generated by the  highest weight vector
$v_{m,n} = e^{- m \delta_1 - n \delta_2}$  where
%$$\delta_1 = \frac{1}{3} ( 2\gamma_1 + \gamma_2), \delta_2 = \frac{1}{3} ( \gamma_1 +  2 \gamma_2). $$
Then $E^{\kappa} [m,n]:=\mathcal W(2,3)_{c}. v_{m,n}$ is a highest weight $\mathcal W(2,3)_{c=-10}$--module with highest weight
  $(h^{(p) } _{m,n} , \beta ^{(p)}_{m,n})$ where
  %{\bf formule su tocne: vidi fajl}
\bea
h^{(p) }  _{m,n} &=& \frac{p}{3}  (m ^2 + n^2 + m n) + (p-1)  (m+n)   \\
 \beta ^{(p)}_{m,n} &=& \frac{(m - n) (-3 + 3 p + 2 m p + n p) (-3 + 3 p + m p + 2 n p)}{2 p^2} .
\eea
 \end{proposition}

 We also have the {\em long} screening operators
 $$ E_1 = e^{\gamma_1}_0, E_2 = e^{\gamma_2}_0$$
 which commutes with  both  $e^{-\gamma_1 /p } _0$ and $e^{-\gamma_2 /p } _0$ \cite{AdM-peking, FT,MP}. Therefore operators $E_1$ and $E_2$ act as derivations on $  \mathcal{W}(p)_{A_2}$. In particular,
 $$ H_1:= E_1 E_2  e^{-\gamma_1 - \gamma_2}, \ \ H_2:= E_2 E_1 e^{-\gamma_1 -\gamma_2}  \in  \mathcal{W}_{A_2} ^ 0 (p) $$
 and they are (non-zero) singular vectors of conformal weight $ 3 p -2$ (e.g. for $p=2$, this weight is $4$).
 
 Note also that for $p=2$:
 $$ [E_1, E_2] =  (e^{\gamma_1}_0 e^{\gamma_2} )_0 =  ( \gamma_1 (-1)  e^{\gamma_1 + \gamma_2} )_0  = - ( \gamma_2(-1)  e^{\gamma_1 + \gamma_2} )_0 . $$

 Let $N_k(\g)$  be the parafermion vertex subalgebra of $L_k (\g)$.

\begin{theorem} \cite{A-TG} For $k=-\frac32$  and $p=2$ we have
$$ N_k (\frak{sl}(3)) = \mathcal{W}^0(p)_{A_2}  .$$
 \end{theorem}

 Now we shall relate $N_k(\frak{sl}(2))$ and $N_k(\frak{sl}(3))$.

 \begin{proposition} For $k=-\frac32$  and $p=2$ we have:
 $$N_k(\frak{sl}(2)) \cong \mathcal{W}^0(p)_{A_2}\cap L_k(\frak{sl}(2))  \subset N_k(\frak{sl}(3)).$$
 Theorefore, $N_k(\frak{sl}(2)) $ is a vertex subalgebra of  $\mathcal{W} ^ 0 (p) _{A_2} $ for $p=2$.
 \end{proposition}
 \begin{proof}
 In \cite{A-TG}, we  realized $V_{\sqrt{p} A_2}$ inside the lattice vertex algebra $V_L$, where
 $L= {\Z} \alpha + {\Z} \beta + {\Z} \delta$ with scalar products
 $$ \langle  \alpha, \alpha \rangle = - \langle \beta, \beta \rangle = \langle \delta, \delta \rangle = 1$$
 (all other scalar products are zero).

 We used the following realization
 $$ \gamma_1 = - 2 \alpha, \gamma_2 = \alpha+ \beta -2\delta, \ h = ( -2 \beta + \delta )(-1)  \in \frak{sl}(2). $$
 Note that
 \begin{equation} \label{zero}
 (-2\beta + \delta, \gamma_1) = (-2 \beta + \delta, \gamma_2) = 0.
 \end{equation}
 This easily implies  $N_k(\frak{sl}(2)) \subset  \mathcal{W}^0 (p) _{A_2} $, and that
  $$\mathcal{W}^0(p)_{A_2}\cap L_k(\frak{sl}(2))  =  N_k(\frak{sl}(2)).$$
 The proof follows.
 \end{proof}

\begin{proposition} \label{conf-emb}
$N_{-3/2} (\frak{sl}(2))$ is conformally embedded in $\mathcal{W}^0(2)_{A_2}$.
\end{proposition}
\begin{proof}
We consider conformal embeddings $L_{-3/2} (\frak{gl}(2))$ into $L_{-3/2}(\frak{sl}(3))$ from \cite{AKMPP-16}. Then the Sugawara Virasoro vector from $L_{-3/2}(\frak{sl}(3))$ coincides with Sugawara Virasoro vector of $L_{-3/2} (\frak{gl}(2))$. Since, Cartan subalgebras $\frak{gl}(2)$ and $\frak{sl}(3)$  have the same rank, we conclude that the parafermion algebra $N_{-3/2} (\frak{sl}(2))$ is conformally embedded in $N_{-3/2} (\frak{sl}(3)) =\mathcal W^0(2)_{A_2}$. The proof follows.
\end{proof}

\section{On the  $\frak{sl}(3)$--action on  $\mathcal{W}(p)_{A_2}$ and its applications}
\label{action-sl3-direct}

A geometric proof of the Lie algebra action on $\mathcal W(p)_{Q}$ is given  by S. Sugimoto in  \cite{Su}.   In this section we shall explore this action in the case $Q=A_2$, $p=2$.  We shall prove that $\mathcal W(2)_{A_2}$ is a simple vertex algebra and  by applying quantum Galois theory we shall prove semi-simplicity of  $\mathcal W(2)_{A_2}$ as $\frak{sl}(3) \times \mathcal W(2,3)_{c=-10}$--module. We shall also reconstruct  explicit formulas for $\g=\frak{sl}(3)$--action.

Our method is based on the following facts:

\begin{itemize}
\item [(1)] As in \cite{Su},  the borel subalgebra action is given by  the screening operators $E_i =e^{\gamma_i} _0$, $i=1,2$.

\item[(2)]  We use formulas from \cite{AdM-2010} for operators $F_i$ such that $F_i$, $E_i$ give $\frak{sl}(2)$-actions on $\mathcal{W}(2)_{A_2}$.

\item[(3)] Using the realization from \cite{A-TG}, and identification of $\mathcal{W}(2)_{A_2}$ as a subalgebra of the vacuum space of $L_{-3/2}(\frak{sl}(3))$,  we construct an automorphism $\Psi$ acting on  $\mathcal{W}(2)_{A_2}$ such that $ F_i = - \Psi   E_i  \Psi^{-1}$.

\end{itemize}

In Section \ref{karakterizacije} we will use these explicit formulas to describe $N_{-3/2}(\frak{sl}(2)$ as $\frak{gl}(2)$ invariants of $W(2)_{A_2}$.

 \vskip 5mm
Recall that the vacuum space is defined as  $$\Omega_k(\g) = \{  v \in L_k(\g) \  \vert \  ({\h} \otimes t {\C}[t] ). v = 0 \}. $$
and it is a generalized vertex algebra \cite{DL}. Moreover,
$$ N_k(\g) \subset \Omega_k(\g). $$
By using the  realization from \cite{A-TG}, we have that $\Omega_{-3/2}(\g)$ is a generalized vertex algebra which contains  vertex subalgebra
$  \mathcal{W}(p)_{A_2}$  for $p=2$.

We shall now use results of \cite{AdM-2010} to construct some derivations of $\mathcal{W}(p)_{A_2}$ for $p=2$.
Let $a^1 = e^{-\frac{1}{2} \gamma_1}$,  $a^2 = e^{-\frac{1}{2} \gamma_2}$. For $i=1,2$, we  define  $$F_i  =  \sum_{j \in {\Z}, j \ne 0 } ^{\infty} \frac{1}{j}  :  a^i _{-j} a^{i} _j : ,  \quad F_i ^{tw}  =  \sum_{j \in  \tfrac{1}{2} + {\Z} } ^{\infty} \frac{1}{j}  :  a^i _{-j} a^{i} _j : .$$

\begin{lemma}
For $i=1,2$, we   have:
\begin{itemize}
\item[(1)]  $F_i $ is a derivation on any $F_i$--invariant  vertex subalgebra $V  \subset {\rm Ker} \ a^i _0$ .
\item[(2)]  $F_i$ is a derivation on  $\Omega_{-3/2}  (\frak{sl}(3))$  and   $\mathcal{W}(2)_{A_2}$.
\item[(3)] $E_i, F_i, h_i:= \frac{1}{2} \gamma_i(0)$ generate an $\frak{sl}(2)$--action on $\Omega_{-3/2}  (\frak{sl}(3))$ and  $\mathcal{W}(2)_{A_2}$.
\end{itemize}
\end{lemma}
\begin{proof}
Assertion (1) follows directly from \cite{AdM-2010}.

Let $V= \Omega_{-3/2}  (\frak{sl}(3))$ or $V= \mathcal{W}_{A_2}(2)$.
 We claim that on $V$ we have
  \bea
  && a^2 _0 F_1  =  F_1 ^{tw} a^2 _0 \label{com-1} \\
  && a^1 _0 F_2  =  F_2 ^{tw} a^1 _0  \label{com-2}.
  \eea
Then assertion (2) would follow directly from relations (\ref{com-1})-(\ref{com-2}).

It remains to prove these relations. Let us prove (\ref{com-1}).  Let $v \in V$.  Note that $a^2$ belongs to a twisted $V_L$--module, and now \cite{AdM-2010} implies
$$ Y(F_1  v, z) a^2 = [F_1^{tw}, Y(v,z) ] a^2. $$
Since $F_1 ^{tw} a^2 = 0$, we get  $Y(F_1  v, z) a^2 = F_1 ^{tw} Y(v,z) a^2$.
Now  skew-symmetry we get
\bea  a^2 _0 F_1 v   &=& \mbox{Res}_z Y(a^2, z) F_1 v \nonumber \\
&=& \mbox{Res}_ z e^{-z L(-1)}  Y(F_1 v,  - z)  a^2 \nonumber \\
&=&  \mbox{Res}_ z e^{-z L(-1)} F_1 ^{tw}  Y( v,  - z)  a^2 \nonumber \\
&=&   F_1 ^{tw} \mbox{Res}_ z e^{-z L(-1)}   Y( v,  - z)  a^2 \nonumber \\
&=& F_1 ^{tw} a^2 _0 v \nonumber \eea
This proves  (\ref{com-1}).  The proof of  (\ref{com-2}) is analogous.

 The assertion (3) follows from  a  direct calculation  as in \cite[Section 4.1]{FFHST}.

\end{proof}

 Now we can reconstruct the $\frak{sl}(3)$ action on $\mathcal W(2)_{A_2}$ (obtained in \cite{Su} by slightly different methods).
 We skip details.

 \begin{itemize}

\item Let $\Psi_1$ be automorphism of the VOA $L_k(\frak{sl}(3))$ lifted from the automorphism  $$\alpha_1 \mapsto - \alpha_2, \alpha_2 \mapsto - \alpha_1 $$
  of the root lattice $A_2$.

  \item  Since $\Omega_k (\frak{sl}(3))$ is $\Psi$--invariant, we conclude that  $\Psi = \Psi_1 \vert \Omega_{-3/2} (\frak{sl}(3))$ is an automorphism of the
 generalized vertex algebra
 $\Omega_k (\frak{sl}(3))$ for every $k$. In particular, $\Psi$   is an automorphism of $ \Omega_{-3/2} (\frak{sl}(3)$ and of  $\mathcal{W}(2)_{A_2}$.

 \item Using realization, we show that $F_i = \Psi^{-1} E_i \Psi$ and $F_{1,2}= \Psi^{-1} E_{1,2} \Psi$, where
  $E_{1,2} = [E_1, E_2]$ and $F_{1,2} =[F_1, F_2]$.

  \item Since $\Psi (\omega) = \omega$ and $\Psi(w_{3}) =-w_{3}$, we see that all derivations above fix $\mathcal W(2,3)_{c=-10}$.
 \end{itemize}

In this way we get an alternative proof of the Sugimoto theorem for $Q=A_2$, $p=2$:

 \begin{theorem} \label{derivation} \cite{Su} The Lie algebra $ \g=\frak{sl}(3)$ acts on $\Omega_{-3/2} (\frak{sl}(3))$ and on $\mathcal{W}(2)_{A_2}$  by derivations. Moreover,
 $$  \mathcal W(2,3)_{c=-10} \subset  \mathcal{W}(2)_{A_2} ^{\g}. $$
 \end{theorem}

 We have the following important consequence:

 \begin{corollary}
 The group $PSL(3,\mathbb{C})$   and the compact Lie group $PSU(3)$ act on $\mathcal{W}(2)_{A_2}$  as  automorphism groups.
The group action commutes with the action of the vertex algebra $\mathcal W(2,3)_{c}$.
 \end{corollary}

  We first need the following simplicity result.

\begin{lemma} \label{simple}
\item[(1)] The vertex algebra $\mathcal W  (2) _{A_2}$ is simple.
\item[(2)] The generalized   vertex algebra $\Omega_{-3/2}(\frak{sl}(3))$ is simple.
\end{lemma}
\begin{proof} We have decomposition of $\mathcal W (2) _{A_2}=\oplus_{\alpha   \in A_2} \mathcal W(2)^{(\alpha )}_{A_2}$ with respect to gradation
in the root lattice $A_2$.
 From \cite{A-TG} we know that $\mathcal W(2) ^{(0)} _{A_2} =N_{-3/2} (\frak{sl}(3))$ is a simple VOA (as the parafermionic algebra of $L_{-3/2}(\frak{sl}_3)$) and
all $\mathcal W (2)^{(\alpha)} _{A_2} $ are also irreducible again, because they are realized as parafermionic modules (see  \cite{CKLR} for detailed analysis of Heisenberg cosets).
% \color{red} koja bi tocno referenca tu isla. Schur-Weyl Heisenberg Cosets? \color{black}).
Moreover, since   each $\mathcal W(2)   ^{(\alpha)}  _{A_2} $ belongs to a simple VOA  $L_{-3/2}(\frak{sl}_3)$, we conclude that $\mathcal W (2)   ^{(\alpha)} _{A_2}  \cdot \mathcal W   (2) ^{(\beta)}   _{A_2}  \ne 0 $ and therefore
  $$\mathcal W  (2)  ^{(\alpha)}  _{A_2}  \cdot \mathcal W    (2) ^{(\beta)}  _{A_2} =\mathcal W(2) ^{(\alpha + \beta)} _{A_2} . $$ This proves the simplicity of $\mathcal W(2)_{A_2}$.

  The proof of simplicity of $\Omega_{-3/2}$ is analogous. We have $\Omega_{-3/2} (\frak{sl}(3)) =\oplus_{\mu   \in P}  \Omega_{-3/2}^{(\mu )}$  with respect to gradation of the weight lattice of $A_2$. Modules  $\Omega_{-3/2}^{(\mu )}$  are simple $\mathcal W ^{(0)} (2) _{A_2} $--modules since they   are realized inside of $L_{-3/2}(\frak{sl}_3)$ as irreducible modules for the parafermionic algebra. This easily proves that  $\Omega_{-3/2}$ is simple.

\end{proof}

\begin{remark}     A different proof  of simplicity of $\Omega_{-3/2}(\frak{sl}(3))$ can be given using more general  results of \cite{Li}. It is not hard to see that $\mathcal{W}  (2) _{A_2} $ is a $\mathbb{Z}_3$-orbifold of $\Omega_{-3/2} (\frak{sl}(3))$, implying also an alternative proof of simplicity of $\mathcal{W}   (2) _{A_2}$.
\end{remark}

 Now  by applying quantum Galois theory   (cf. \cite{DLM}, \cite{DG}, \cite{DM}) we get:
\begin{proposition} \label{comp-red-galois}
Let $\g = {\frak sl}(3)$.
The vertex algebra  $\mathcal W (2) _{A_2} $  is  a completely   reducible $ \g \times \mathcal W  (2) _{A_2} ^{\g}$--module and
  $$ \mathcal W  (2) _{A_2} = \bigoplus _{n,m \geq 0, \atop n \equiv m  \mod (3) }    V_{A_2} (m \omega_1 +  n \omega_2) \otimes L[m, n] $$
  where $L[m,n]$ is  certain irreducible $\mathcal W  (2) _{A_2} ^{\g}$--module.
\end{proposition}
\begin{proof}
Using \cite{DLM}, and more precisely \cite[Remark 2.3]{DG},   and Lemma \ref{simple}, we first get a decomposition of $\mathcal{W} (2) _{A_2}$
as an $\g  \times \mathcal{W} (2) _{A_2} ^{\g}$ -module:
$$\bigoplus_{\lambda \in P'} V_{\frak{sl}(3)}(\lambda) \otimes L_{\lambda}$$
where $P'$ is a set that contains $P^+ \cap Q$ and $L_{\lambda} $ are irreducible $\mathcal{W} (2) _{A_2} ^{\g}$-modules. From the definition of  the $\g$--action on $\mathcal{W} (2) _{A_2}$ we easily see that
any weight in $P'$ is necessarily inside $\lambda \in P^+ \cap Q$. We conclude that $\mathcal{W} (2) _{A_2}$ decomposes as
$$\bigoplus_{\lambda \in P^+ \cap Q} V_{\g}( n \omega_1+m \omega_2  ) \otimes L [m,n].$$
The proof follows
\end{proof}

\section{ $\mathcal W^0(2)_{A_2}$ and $N_{-3/2}(\frak{sl}(2))$ as invariant subalgebras }
\label{karakterizacije}
In this section we shall identify $\mathcal W^0(2)_{A_2}$ as the parafermion subalgebra of the $N=4$ superconformal vertex algebra $L^{N=4}_{c=-9} = \mathcal V^{(2)}$ (cf. \cite{A-TG}). By using identification of $L_{-3/2}(\frak{sl}(2))$ as the $\frak{sl}(2)$ invariant subalgebra of $\mathcal V^{(2)}$ we shall prove that the parafermion algebra $N_{-3/2}(\frak{sl}(2))$ is the $\frak{gl}(2)$--invariant subalgebra of $\mathcal W(2)_{A_2}$.

We shall consider two subalgebras of the Lie algebra $\g=\frak{sl}(3)$ for which acts on $\mathcal W(2)_{A_2}$ by derivations.

Let $ \g_0 = \mbox{span}_{\C} \{ E_1, F_1, h_1 \} \subset \g$  and   $\g_1 = \g_0 + {\C} h_2 \subset \g$. Clearly, $\g_0 \cong \frak{sl}(2)$ and $\g_1 \cong \frak{gl}(2)$.

Note that $\g_0$  acts on the $N=4$ superconformal vertex algebra $\mathcal V^{(2)}$ (cf. \cite{A-TG}, \cite{ACGY})  and
$ L_{-3/2} (\frak{sl}(2)) = \left( \mathcal V^{(2)}   \right)^{\g_0}$.

The vertex superalgebra $\mathcal V^{(2)}$ admits a ${\Z}$--gradation:
 $$ \mathcal V^{(2)} = \bigoplus_{\ell \in {\Z}}  \mathcal V^{(2)} _{\ell}, $$
 such that $  \mathcal V^{(2)} _0 $ is its subalgebra which decomposes as follows:
 \bea  \mathcal V^{(2)} _0  = \bigoplus _{s=0} ^{\infty} L_{\frak{sl}(2)} ( 2 s \omega_1). \label{dec-v0} \eea

 Let $M_h(1)$ is the Heisenberg vertex algebra generated by $h$.

\begin{proposition} \label{invarijant-1} We have:
\begin{itemize}
\item[(1)] $ \mathcal{W}^0(2)_{A_2}  =  Com (M_h(1), W)$, and  the following decomposition holds
 $$  \mathcal{W}^0(2)_{A_2} = \bigoplus _{s=0} ^{\infty} N_{\frak{sl}(2)} ( 2 s). $$
\item[(2)]  $N_{-3/2} (\frak{sl}(2)) \cong \mathcal W(2)^{\g _1}_{A_2}. $
\end{itemize}
\end{proposition}
\begin{proof}
In the realization presented in \cite{A-TG}, $L_{-3/2} (\frak{sl}(3))$ is realised as a subalgebra of tensor product $\mathcal V^{(2)}  \otimes F_{-1}$, where
where $F_{-1}$ is a lattice vertex algebra.  Moreover,
$  \mathcal{W}^0(2)_{A_2}  = N_{-3/2} (\frak{sl}(3))$ is exactly  the subalgebra of $\mathcal V^{(2)} _0 $  on which $h(n)$, $n \ge 0$,  act trivially.  Using (\ref{dec-v0}) we get
$$  \mathcal{W}^0(2)_{A_2} = \bigoplus _{s=0} ^{\infty} N_{\frak{sl}(2)} ( 2 s). $$
This proves assertion (1).

Consider next the parafermion  vertex algebra of  $\mathcal V^{(2)}$
$$ \mathcal  U ^{(2)} = \mbox{Com} (M_h(1), \mathcal V^{(2)} ) = \{  v \in \mathcal V^{(2)}   \vert \ h(n) v = 0,  n \ge 0 \} \subset \Omega_{-3/2}. $$
Since the action of $\g_0$ commutes with operators $h(n)$ we conclude that $\g_0$ acts on
 $ \mathcal U^{(2)}  $ and we have the following decomposition of $\mathcal U^{(2)} $ as a $ \g_0 \times N_{-3/2} (\frak{sl}(2))$--module
 $$ \mathcal U^{(2)}  = \bigoplus_{j \in {\Z}_{\ge 0} } \rho_{2j} \otimes N_{-3/2} (2j) $$
 where $\rho_{j } $ denotes the irreducible $j+1$--dimensional $\g_0$--module.
 Moreover, we have
 $$ N_{-3/2} (\frak{sl}(2)) =  \left ( \mathcal U ^{(2)}   \right)^{\g_0} \subset  \mathcal W_{A_2} (2) ^{\g _1} . $$
Since
 $$ \mathcal W (2) ^{\g _1}_{A_2}  \subset  \mathcal W ^0(2) _{A_2}  = \mbox{Com} (M_h(1), \mathcal V^{(2)} _0 ) \subset \mathcal U^{(2)} $$
 we have $N_{-3/2} (\frak{sl}(2)) \cong \mathcal W (2) ^{\g _1}_{A_2}$. The proof follows.
\end{proof}

 \section{The character of the parafermion vertex algebra $N^k (\frak{sl}(2))$ }
\label{char-nk}
 Let  $(x; q)_n = \prod_{i=0} ^{n-1} (1-xq^i)$ and $(q)_n=(q;q)_n$. For $m \leq 0$, let
  $$ \Phi_m (q) = \sum_{r=0} ^{\infty} (-1) ^r q ^{\frac{r (r+1)}{2}+m r},$$
  and for $m>0$, $\Phi_m(q)=\Phi_{-m}(q)$ denote unary false theta functions.
 By an identity of Andrews \cite{Andrews} (see also \cite{B-H,BKMZ}),  the character of $N^k(\frak{sl}(2))$ can be computed as
 $$ {\rm ch}[ N^k (\frak{sl}(2))](q) ={\rm CT}_{x} \frac{1}{(xq;q)_\infty (x^{-1}q;q)_\infty} = \frac{ \Phi_0(q) - \Phi_{-1}(q) }{(q)_\infty^2},$$
 where ${\rm CT}_x$ denotes the constant term inside the range $|q|<|x|<1$.

  We have the following $q$--character for the  vertex algebra  $N^k (\frak{sl}(2))$ for $k=-\frac32$. Although the same formula
  is valid for any generic level, this specific shape is convenient for $k=\frac32$.
  \begin{lemma} \label{2.1} We have:
$$
 {\rm ch}[N^k (\frak{sl}(2))](q) =  \frac{    \sum _{m = 1} ^{\infty} q^{2m (m-1)}  ( 1 - q^m) (1-q^m) (1-q^{2m}) }{(q)_\infty^2}.
  $$
  \end{lemma}
\begin{proof} We have:
   \bea
  && \sum _{m = 1} ^{\infty} q^{2m (m-1)}  ( 1 - q^m) (1-q^m) (1-q^{2m})  \nonumber \\
= &&  \sum _{m = 1} ^{\infty} q^{2m (m-1)}  ( 1  - 2 q^m + 2 q^{3m} - q^{4m} )   \nonumber  \\
= &&   \sum _{m = 1} ^{\infty}  (    q^{2m (m-1)} - 2  q^{m (2 m-1)} +  2  q^{m (2 m+ 1)}  -  q^{2 m (m+1) } )  \nonumber  \\
= &&  \sum _{m = 1} ^{\infty}  (    q^{2m (m-1)}  -  q^{2 m (m+1) } )   +  \sum _{m = 1} ^{\infty}  (     - 2  q^{m (2 m-1)} +  2  q^{m (2 m+ 1)}  )   \nonumber \\
= && 1 - 2 \sum _{m = 1} ^{\infty}  (       q^{m (2 m-1)} -  q^{m (2 m+ 1)}  )  \nonumber \\
= && 1 + 2 \sum _{i  = 1} ^{\infty}     (-1) ^i     q^{ \frac{i (i+1)}{2} }    \nonumber  \\
= && \sum _{i  = 0 } ^{\infty}     (-1) ^i     q^{ \frac{i (i+1)}{2} } -  \sum _{i  = 0 } ^{\infty}     (-1) ^i     q^{ \frac{i (i+1)}{2} -i  }   \nonumber  \\
= && \Phi_0(q) - \Phi _{-1}(q) . \nonumber
\eea
The proof follows.
\end{proof}

\begin{lemma} We have:
\bea
 {\rm ch}[N^k (2 s) ] (q) &= & q^{2s(s+1) } {\rm CT} \frac{\left(x^{-s}+\cdots +1 + \cdots + x^s \right)}{(xq;q)_\infty (x^{-1}q;q)_\infty} \nonumber \\
 &=&  \frac{q^{2s(s+1) }\left( \Phi_{0}(q)+\Phi_{-1}(q)-2\Phi_{-s-1}(q) \right)}{(q;q)_\infty^2 }.
\nonumber
 \eea
\end{lemma}

  The following lemma can be viewed as a  generalization  of Lemma \ref{2.1}.
\begin{proposition} \label{par-char} We let
$$F_{m,n}=q^{\frac23m^2+\frac23n^2+\frac23mn-m-n}(1-q^m)(1-q^n)(1-q^{m+n}).$$
Then for every $s \geq 0$, we have
$$\underbrace{\sum_{m \geq s+1 } F_{m,m}+\sum_{1 \leq i \leq s, m \geq i} (F_{m,m+3(s+1-i)}+F_{m+3(s+1-i),m})}_{:=
G_s(q)}= (q;q)_\infty^2 {\rm ch}[N^k (2 s) ] (q). $$
%q^{2s(s+1)} \left(\Phi_0(q)+\Phi_{-1}(q)-2\Phi_{-s-1}(q)\right)$$
\end{proposition}

%{\bf A: E sada, taj $2s(s+1)$ je tocno tezina of $f_s(0)v_s$ za koset Virasoro. ok?}{\color{blue} OK}

\begin{proof}
Follows by direct computation by induction. For $s=0$ this is the statement of Lemma \ref{2.1}.
For the induction step we observe (for $s \geq 0$)
\begin{align*}
& (q;q)_\infty^2 {\rm CT} \frac{\left(x^{-s-1}+\cdots +1 + \cdots + x^{s+1} \right)}{(xq;q)_\infty (x^{-1}q;q)_\infty} \\
&= (q;q)_\infty^2 {\rm CT} \frac{\left(x^{-s}+\cdots +1 + \cdots + x^{s} \right)}{(xq;q)_\infty (x^{-1}q;q)_\infty}+2 (\Phi_{-s-1}(q)-\Phi_{-s-2}(q)).
\end{align*}
The rest follows via manipulation with $q$-series as in Lemma \ref{2.1}.
\end{proof}

\begin{proposition}
$$\sum_{m,n \geq 1, m \equiv n (3)} \min(m,n)q^{\frac23m^2+\frac23n^2+\frac23mn-m-n}(1-q^m)(1-q^n)(1-q^{m+n})= \sum_{s \geq 0} G_{s}.$$
\end{proposition}
\begin{proof}
Directly from definition of $G_s$ and observation that every $F_{m,n}$ appears exactly $\min(m,n)$-times inside $G_s$ as a summand.

The formula also follows from Proposition \ref{invarijant-1}.
\end{proof}

\color{black}

We also record a $q$-hypergeometric expression for the character.
\begin{proposition} We have
$${\rm ch}[N^k (\frak{sl}(2))](q)=\sum_{n \geq 0} \frac{q^{2n}}{(q)_n^2}.$$
where $(q)_n=\prod_{i \geq 1}^n (1-q^i)$.
\end{proposition}
\begin{proof} It follows from Euler's identity $$\frac{1}{(x;q)_\infty}=\sum_{n \geq 0} \frac{q^n x^n}{(q)_n}$$
and the fact that ${\rm ch}[N^k (\frak{sl}(2))](q)$ is the constant term of $\frac{1}{(x;q)_\infty (x^{-1};q)_\infty}.$
\end{proof}

Lemma \ref{2.1} and Lemma \ref{irr-upper} suggest the following result whose proof is postponed
for Section
\ref{dec-parafermion-1}.

\begin{theorem} \label{dec-sl2} Let $k =-\frac32$. As a ${\mathcal W(2,3)}_{c=-10}$--module, we have
\bea N_k(\frak{sl}(2)) \cong  \bigoplus_{m=1} ^{\infty}  L^{W}(-10, 2 m (m-1), 0). \label{dec-w23} \eea
\end{theorem}

 \section{ The Character of $\mathcal{W}^0(2)_{A_2}$}

 In this section we discuss two formulas for the character of $\mathcal{W}^0(2)_{A_2}$ from \cite{BKMZ} (see also \cite{BKM}).

Using the realization of $\mathcal{W}(2)_{A_2}$ from \cite{A-TG}, the character of $\mathcal{W}^0 (2) _{A_2}$ can be computed in an elegant form. \begin{theorem} \cite{BKMZ} \label{BKMZ} We have
 \begin{align*}
&  {\rm ch}[\mathcal{W}^0(2)_{A_2}](q) \\
& =\frac{\displaystyle{\sum_{m,n \geq 0 \atop m=n (mod \ 3)} }{\rm min}(m+1,n+1) q^{\frac23m^2+\frac23 mn + \frac23n^2+m+n }(1-q^{m+1})(1-q^{n+1})(1-q^{m+n+2})  }{(q)_\infty^2}.
\end{align*}
 \end{theorem}
 Observe that if we sum over the subset $m=n \in \mathbb{N}$ above, we obtain the character of $N^k(\frak{sl}(2))$ as a summand.
   This important observation will be explained in  Section \ref{dec-A2-w3}.

Results from Section \ref{karakterizacije} can be used to give a new representation-theoretic proof
of  the following result from \cite{BKMZ}.
\begin{proposition}
$${\rm ch}[\mathcal{W}^0(2)_{A_2}](q)=\frac{\sum_{n_1 \geq 0, n_2 \in \mathbb{Z}} {\rm sgn}(n_2)(-1)^{n_1} q^{\frac{n_1(n_1+1)}{2}+n_1n_2+2n_2^2+2n_2}}{(q;q)_\infty^2},$$
where ${\rm sgn}(n)=1$, $n \geq 0$ and $-1$ for $n<0$.
\end{proposition}
\begin{proof}
Using Proposition \ref{invarijant-1} we get:
$$ \mathcal W^0 (2)_{A_2} = \bigoplus_{s=0} ^{\infty} N_{-3/2} (2 s).$$
Applying the character formulas for the relevant $\widehat{\frak{sl}(2)}$-modules, we can write
\begin{equation*}
{\rm ch}[\mathcal{W}_{A_2}^0(2)](q)={\rm CT}_x \frac{\sum_{s \geq 0} \frac{x^{s+1/2}-x^{-s-1/2}}{x^{1/2}-x^{-1/2}} q^{2s(s+1)}}{(xq;q)_\infty (x^{-1}q;q)_\infty}.
\end{equation*}
The rest follows simply by extracting the constant term, using ($m \in \mathbb{Z}$)
$${\rm Coeff}_{x^{m}}  \sum_{s \geq 0} \frac{x^{s+1/2}-x^{-s-1/2}}{x^{1/2}-x^{-1/2}} q^{2s(s+1)}=\sum_{s \geq |m|} q^{2s(s+1)} $$
and ($m \geq 0$)
$${\rm Coeff}_{x^{m}} \frac{1}{(xq;q)_\infty (x^{-1}q;q)_\infty}=\frac{1}{(q;q)^2_\infty} \left(\Phi_{-m}(q)-\Phi_{-m-1}(q)\right),$$
discussed in Section \ref{char-nk}. Finally, we have to split the numerator in the character formula as
$$\sum_{n_1, n_2 \geq 0} (-1)^{n_1} q^{\frac{n_1(n_1+1)}{2}+n_1n_2+2n_2^2+2n_2}
 - \sum_{n_1, n_2 \geq 0} (-1)^{n_1} q^{\frac{n_1(n_1+1)}{2}-n_1(n_2+1)+2n^2+2n_2}.$$
\end{proof}

 \section{The decomposition of $\mathcal W(2)_{A_2} $ as $\mathcal{W}(2,3)_c$--module}

 \label{dec-A2-w3}

We also require next computational lemma.
\begin{lemma} \label{hw} For every $m,n \geq 0$ such that  $m \omega_1 + n \omega_2 \in P^+ \cap Q$,  $v_{m,n}=e^{-m \omega_1-n \omega_2}   \in \mathcal{W}(2)_{A_2}$. Also,
%{\bf AM: zasto $-$ tu treba biti +. Uzmi $m=n=1$ i ispada da je tezina nule.}
$$L(0) \cdot v_{m,n}=h_{m,n} v_{m,n},  \quad W(0) \cdot v_{m,n}=\beta_{m,n} v_{m,n},$$
where  $h_{m,n} := h_{m,n} ^{(2)}  =\frac23 m^2+\frac23n^2+\frac23mn +m+n  $, $\beta_{m,n}:= \beta_{m,n} ^{(2)} = \frac{1}{8}  (m - n) ( 3 + 4 m + 2 n) (3 + 2 m + 4 n)$.
%\frac{4}{27}(\frac23m-\frac{n}{3}+\frac13)(\frac23n-\frac{m}{3}+\frac13)(\frac{m}{3}-\frac{n}{3}+\frac13)$.
%Moreover, $v_{m,n}$ is a highest weight vector for the $\frak{sl}(3)$ acting on ${W}(2)_{A_2}$.
\end{lemma}
\begin{proof} From the definition, $v_{m,n}=e^{-m \gamma_1-n \gamma_2}$. Under the imposed condition on $m$ and $n$, these vectors are annihilated by the screening operators $e^{-\frac12 \gamma_1}_0$ and  $e^{-\frac12 \gamma_2}_0$ and therefore $v_{m,n} \in \mathcal{W}(2)_{A_2}$.  Computation of highest weights   follows from Proposition \ref{T}.
%is straightforward using explicit formulas for $\omega$ and $W$ (this was also computed in \cite{S1}).
\end{proof}

\begin{theorem}  \label{main1}Let $\g=\frak{sl}(3)$, we have:
\item[(1)] $\mathcal W(2)^{\g}_{A_2}  \cong \mathcal W(2,3)_{c=-10}$,
\item[(2)] $$\mathcal{W}(2)_{A_2} \cong \bigoplus_{n,m \geq 0, \atop n \equiv m (mod \ 3)} V_{\frak{sl}(3)}(n \omega_1+m \omega_2) \otimes L^W \left(-10, h_{m,n} ,\beta_{m,n} \right), $$
\item[(3)]  \begin{equation} \label{dec}
\mathcal{W}^0(2)_{A_2} \cong \bigoplus_{n,m \geq 0, \atop n \equiv m (mod \ 3)}  {\rm min}(m+1,n+1) L^W \left(-10, h_{m,n} ,\beta_{m,n}\right).
\end{equation}
\end{theorem}

\begin{proof}
First we notice that $ \mathcal W(2)^{\g}_{A_2}  \subset \mathcal W^0 (2) _{A_2}$, and therefore   $\mathcal W^0 (2) _{A_2}$ is a completely reducible  $ \mathcal W (2)^{\g}_{A_2}$--module.
The contribution of an irreducible $\frak{sl}(3)$-module in $\mathcal{W}^{0}(2)_{A_2}$ is controlled by the weight zero subspace
whose dimension is given by
$${\rm dim}(V_{\frak{sl}(3)}(m \omega_1+n \omega_2)_0)=\min (m+1,n+1), \ \ m \omega_1+n  \omega_2 \in Q \cap P^+.$$
Thus we have
\bea  \mathcal W^ 0 (2)_{A_2} = \bigoplus _{n,m \geq 0, \atop n \equiv m  \mod (3) }   \min (m+1,n+1)  L[m, n]. \label{dec-wp0} \eea
From Theorem \ref{BKMZ}, we get
$${\rm ch}[\mathcal W^0 (2)_{A_2}](q)=\sum_{n,m \geq 0, \atop n \equiv m (mod \ 3)}  {\rm min}(m+1,n+1) {\rm ch}[T^\kappa_{m \omega_1+n \omega_2,0}](q), $$
 implying that  ${\rm ch}[T^\kappa_{m \omega_1+n \omega_2,0}](q)$ is the $q$--character of certain irreducible $\mathcal W(2)^{\g}_{A_2}$--module.
Since
\begin{itemize}
\item $\mathcal W(2,3)_{c=-10} \subset \mathcal W (2)^{\g}_{A_2}$,
\item ${\rm ch} [\mathcal W(2,3)_{c=-10}](q) = {\rm ch}[T^\kappa_{0,0}](q) <  {\rm ch} [\mathcal W(2)^{\g}_{A_2}](q),  $
\end{itemize}
we conclude that  $\mathcal W(2)^{\g}_{A_2} \cong \mathcal W(2,3)_{c=-10}$.  This proves the assertion (1). Therefore all $L[m,n]$ are irreducible $\mathcal W(2,3)_{c=-10}$--modules. Lemma \ref{hw} implies that
$$L[m,n]= L^{W} (-10, h_{m,n}, \beta_{m,n}). $$
Then assertion (2) follows from  Proposition  \ref{comp-red-galois}, and assertion (3) from the decomposition (\ref{dec-wp0}).
\end{proof}

From the previous result we also easily get a proof of Theorem \ref{dec-sl2}.
Also, as a consequence, we obtain character formulas for irreducible $W(2,3)$-modules considered in
Lemma \ref{irr-upper}.
\begin{corollary} \label{irr-WW} For $m,n \geq 0$, we have
$${\rm ch}[L^W(-10, \frac23m^2+\frac23 mn + \frac23n^2+m+n,\beta_{m,n} )](q)$$
$$=\frac{q^{\frac23m^2+\frac23 mn + \frac23n^2+m+n }(1-q^{m+1})(1-q^{n+1})(1-q^{m+n+2})  }{(q;q)_\infty^2}.$$
\end{corollary}
\begin{proof} For $m = n \mod 3$ this follows directly from Theorem \ref{main1} and Lemma \ref{irr-upper}.
For $m \neq n \mod 3$, this is a consequence of the same analysis applied to $\mathcal{W}(2)_{A_2}$-modules. We omit details
here.
\end{proof}

 \begin{remark}\label{koos} We should point out that irreducible $W(2,3)$-modules and their characters
were studied in an old work of Koos and Driel \cite{KV}. In particular, they analyzed irreducible representations with central charge $c=50-24p-\frac{24}{p}$, for $p \in \mathbb{N}$, parametrized by dominant integral weights.

For $p=1$ (i.e. $c=2$), the structure of such representations is well-known due to unitarity of the bosonic construction.  For $p \geq 3$, they proposed explicit formulas for
characters based on summation over the finite Weyl group, subject to conjectural embedding formulas
among Verma modules controlled by certain double cosets in the affine Weyl group of $\frak{sl}(3)$. These embeddings formulas are recently (rigorously) proven by by Dhillon \cite{Gurbir}. His result supposedly clarifies the character formulas used in \cite{KV} (and also in \cite{FT} for other higher rank algebras). However, for $p=2$ (i.e. $c=-10$)  the embedding structure
among Verma modules \cite[Table VIII]{KV} is different compared to $p=3$ (i.e. $c=-30$); specifically \cite[Table IX]{KV}. Even though  Table VIII yields correct formulas for ${\rm ch}[L^W(-10,0,0)](q)$ and ${\rm ch}[L^W(-10,4,0)](q)$, we believe that our Corollary \ref{irr-WW} gives the first rigorous derivations of characters in all cases.
\end{remark}

  \section{The generators of $\mathcal{W}^{0} (2) _{A_2}$}

  \begin{theorem} We have:
  \item[(1)]  The vertex algebra $\mathcal{W}^{0} (2)_{A_2}$  is generated by $N_{-3/2}(\frak{sl}(2)) + N_{-3/2} (2)$.

  \item[(2)] The vertex algebra $\mathcal{W}^{0} (2)_{A_2}$ is a simple $\mathcal W$-algebra generated by the conformal
  vector and three primaries of weight $3,4$ and $4$.
  \end{theorem}

  \begin{proof}    Consider again the subalgebra
  $$\mathcal V^{(2)} _0 = \bigoplus_{\ell = 0} ^{\infty} W_{2 \ell}, \quad W_{2\ell} = L_{-3/2} (2 \ell \omega_1)$$ of the $N=4$ superconformal algebra with central charge $c=-9$.
  Using the $\frak{sl}(2)$  action,  quantum Galois theory,  and the same arguments as in the \cite[Lemma 2.6]{DG}, we get
  $$ W_2 \cdot W_{2\ell} = W_{2\ell+2} + W_{2\ell-2} \quad (\ell \ge 1),$$
 where the dot product is defined
 $$U \cdot V={\rm span} \{ u_m v :  u \in U , v \in V, m \in \mathbb{Z} \}.$$
But restriction to the parafermion algebra gives that
 $$ N_{-3/2} (2)  \cdot N_{-3/2} (2\ell) = N_{-3/2} (2\ell+2)  +  N_{-3/2} (2\ell-2)\quad (\ell \ge 1).$$
  This proves assertion (1).

  (2) By Part (1), $\mathcal{W}^0(2)_{A_2}$ is generated by primary vectors of degree $2,3,4,4,5$. However it is easy to see that
  the weight $5$ primary vector can be expressed using other primaries.

  \end{proof}

\begin{remark}
We expect that $\mathcal{W}(2)_{A_2}$ is of type $(2,3,4^8)$.
\end{remark}

\section{ The decomposition of $N_{-3/2} (\frak{sl}(2))$ as a $\mathcal W(2,3)_{c=-10}$--module}

\label{dec-parafermion-1}

 We also need the following result which easily follows from \cite[Chapter 8]{GW}.

 \begin{lemma} \label{invarijante-2} For every $m, n \in {\Z}_{\ge 0}$, we have
 $$ \dim V_{\frak{sl}(3)} (n \omega_1 + m \omega_2) ^{\g_1}  =  \delta_{n,m}.$$
 \end{lemma}

\begin{theorem} \label{dec-sl2-10} Let $k =-\frac32$. As a ${\mathcal W(2,3)}_{c=-10}$--module, we have
\bea N_k(\frak{sl}(2)) \cong  \bigoplus_{m=0} ^{\infty}  L^{W}(-10, 2 m (m+1), 0). \label{dec-w23-10} \eea
\end{theorem}

\begin{proof}
Using Proposition \ref{invarijant-1}, Lemma  \ref{invarijante-2}  and Theorem  \ref{main1}  we get:
\bea
N_{-3/2} (\frak{sl}(2)) &=&   \mathcal W_{A_2} (2) ^{\g _1} \nonumber \\
&=&   \bigoplus_{n,m \geq 0, \atop n \equiv m (mod \ 3)}   V_{\frak{sl}(3)} (n \omega_1 + m \omega_2 )^{\g_1}  \otimes  L^W \left(-10, h_{m,n} ,\beta_{m,n}\right) \nonumber \\
&=&    \bigoplus_{m \geq 0 }     L^W \left(-10, h_{m,m} ,\beta_{m,m }\right) \nonumber \\
&=&  \bigoplus_{m=0} ^{\infty}  L^{W}(-10, 2 m (m+1), 0). \nonumber
\eea
The proof follows.
\end{proof}
%\color{red}
\begin{remark} One can also give decompositions of  $N_{-3/2} (\frak{sl}(2))$-modules $N_{-3/2}(2 \ell)$, $\ell \geq 1$. The relevant irreducible $W(2,3)_{c=-10}$-modules can be
read off from Proposition \ref{par-char}.
\end{remark}
%\color{black}

\begin{remark} We expect that for arbitrary $p$ the vertex algebra
 $\mathcal W (p)_{A_2}  ^{\frak{gl}(2)}$ is a simple $\mathcal W$--algebra of type $\mathcal W(2,3,3p-2, 3p-1)$.
 \end{remark}

%}

\vskip10pt {\footnotesize{}{ }\textbf{\footnotesize{}D.A.}{\footnotesize{}:
Department of Mathematics, Faculty of Science,  University of Zagreb, Bijeni\v{c}ka 30,
10 000 Zagreb, Croatia; }\texttt{\footnotesize{}adamovic@math.hr}{\footnotesize \par}

\vskip 10pt {\footnotesize{}{}\textbf{\footnotesize{}A.M.}{\footnotesize{}:
Department of Mathematics and Statistics, SUNY Albany, 1400 Washington Avenue,
Albany NY 12222, USA; }\texttt{\footnotesize{}amilas@albany.edu}{\footnotesize \par}

\vskip 10pt \textbf{\footnotesize{}Q.W.}{\footnotesize{} School of Mathematical
Sciences, Xiamen University, Fujian, 361005, China;} \texttt{\footnotesize{}
qingwang@xmu.edu.cn}{\footnotesize \par}

\end{document}